\documentclass[10pt]{article}
	
\usepackage{amsmath}
\usepackage{amsfonts}

\usepackage{amsthm}

\usepackage{hyperref}

\usepackage{cleveref}

\usepackage{mathtools}

\usepackage{mathrsfs}

\usepackage{pdfpages}

\usepackage{microtype}

\newtheorem{theorem}{Theorem}[section]
\newtheorem{proposition}[theorem]{Proposition}
\newtheorem{corollary}[theorem]{Corollary}
\newtheorem{lemma}[theorem]{Lemma}
\newtheorem{definition}[theorem]{Definition}

\newtheorem*{unnumberedtheorem}{Theorem}
\newtheorem*{unnumberedproposition}{Proposition}
\newtheorem*{unnumberedcorollary}{Corollary}

\newtheoremstyle{remarkstyle}
{}
{}
{}
{}
{\bfseries}
{.} 
{7pt}
{}

\theoremstyle{remarkstyle}
\newtheorem{remark}[theorem]{Remark}
\newtheorem*{acknowledgements}{Acknowledgements}

\newcommand{\N}{\mathbb{N}}
\newcommand{\Z}{\mathbb{Z}}

\newcommand{\R}{\mathbb{R}}
\newcommand{\C}{\mathbb{C}}

\DeclareMathOperator{\Rep}{Rep}
\renewcommand{\H}{\mathbb{H}}
\newcommand{\TeichSpace}{\mathcal{T}}

\DeclareMathOperator{\SL}{SL}

\DeclareMathOperator{\PSL}{PSL}

\DeclareMathOperator{\PSO}{PSO}

\DeclareMathOperator{\id}{id}

\DeclareMathOperator{\vol}{vol}

\DeclareMathOperator{\Area}{Area}
\newcommand{\Energy}{\mathscr{E}}
\DeclareMathOperator{\Imag}{Im}
\DeclareMathOperator{\Real}{Re}

\newcommand{\parder}[2]{\frac{\partial #1}{\partial #2}}

\DeclarePairedDelimiter{\norm}{\lVert}{\rVert}
\DeclarePairedDelimiter{\abs}{\lvert}{\rvert}
\DeclarePairedDelimiterX{\inner}[2]{\langle}{\rangle}{#1, #2}

\title{The energy spectrum of metrics on surfaces}
\author{Ivo Slegers}
	
\begin{document}
\maketitle
	
\begin{abstract}
Let $(N,\rho)$ be a Riemannian manifold, $S$ a surface of genus at least two and let $f\colon S \to N$ be a continuous map. We consider the energy spectrum of $(N,\rho)$ (and $f$) which assigns to each point $[J]\in \TeichSpace(S)$ in the Teichm\"uller space of $S$ the infimum of the Dirichlet energies of all maps $(S,J)\to (N,\rho)$ homotopic to $f$. We study the relation between the energy spectrum and the simple length spectrum. Our main result is that if $N=S$, $f=\id$ and $\rho$ is a metric of non-positive curvature, then the energy spectrum determines the simple length spectrum. Furthermore, we prove that the converse does not hold by exhibiting two metrics on $S$ with equal simple length spectrum but different energy spectrum. As corollaries to our results we obtain that the set of hyperbolic metrics and the set of singular flat metrics induced by quadratic differentials satisfy energy spectrum rigidity, i.e. a metric in these sets is determined, up to isotopy, by its energy spectrum. We prove that analogous statements also hold true for Kleinian surface groups.
\end{abstract}

\section{Introduction}
In this paper we study, what we will call, the energy spectrum of a Riemannian manifold (see \Cref{sec:theenergyspectrum}). Let $S$ be a closed surface of genus at least two, let $\TeichSpace(S)$ be its Teichm\"uller space, let $(N,\rho)$ be a Riemannian manifold and let $[f]$ be a homotopy class of maps $S\to N$. In brief, the energy spectrum of $(N,\rho)$ and $[f]$ is the function on Teichm\"uller space that assigns to each $[J] \in \TeichSpace(S)$ the infimum of the energies of all Lipschitz maps $(S,J) \to (N,\rho)$ that lie in $[f]$. It gives a measure of how compatible $(N,\rho)$ and a point in Teichm\"uller space are.

The energy spectrum has been considered (under a different name\footnote{In \cite{LabourieCrossRatios} and \cite{Toledo} it is called the energy function or energy functional.}) by several authors. Toledo proved in \cite{Toledo} that the energy spectrum (for any $[f]$) is a plurisubharmonic function on Teichm\"uller space if $(N,\rho)$ is a compact manifold of non-positive Hermitian curvature. He used this result to give an alternative formulation of the rigidity theory of Siu and Sampson. In \cite{LabourieCrossRatios} Labourie used the energy spectrum to study Hitchin components in representation varieties. Given a Hitchin representation $\rho \colon \pi_1(S) \to \PSL(n,\R)$ he considered the energy spectrum of $N = \rho(\pi_1(S))\setminus \PSL(n,\R)/\PSO(n)$ and the homotopy class of maps that lift to $\rho$-equivariant maps $\widetilde{S} \to \PSL(n,\R)/\PSO(n)$. He proved that it is a proper function on Teichm\"uller space. Furthermore, he made the conjecture that it has a unique minimum. The author showed in \cite{Slegers2} that in this same setting the energy spectrum is strictly plurisubharmonic.

In this paper we examine to what extend a Riemannian manifold is determined by its energy spectrum. We begin by restricting ourselves to the case $N=S$ and $[f] = [\id]$. We will define, by analogy with simple length spectrum rigidity, the notion of energy spectrum rigidity. We will say a set $\mathcal{M}$ of metrics on $S$, determined up to isotopy, satisfies energy spectrum rigidity if the map $\mathcal{M}\to C^0(\TeichSpace(S))$, assigning to each metric its energy spectrum, is an injection. We will study the question which sets of metrics satisfy this type of rigidity.

The main results of this paper offer a comparison between the energy spectrum and the simple length spectrum. Our first result states that the energy spectrum determines the simple length spectrum.
\begin{unnumberedtheorem}[{\Cref{thm:energyspectrumgiveslengthspectrum}}]
Let $\rho, \rho'$ be non-positively curved Riemannian metrics on a surface $S$ of genus at least two. If the energy spectra of $(S,\rho)$ and $(S,\rho')$ (with $[f]=[\id]$) coincide, then the simple length spectra of $\rho$ and $\rho'$ coincide.
\end{unnumberedtheorem}
Our second second results shows that the converse is not true. Namely, the energy spectrum carries strictly more information and hence is not determined by the simple length spectrum.
\begin{unnumberedproposition}[{\Cref{prop:lengthspectrumdoesnotgiveenergyspectrum}}]
For every hyperbolic metric on a surface there exists a negatively curved Riemannian metric on that surface with equal simple length spectrum but different energy spectrum.
\end{unnumberedproposition}

In summary, the energy spectrum is a strictly more sensitive way to tell metrics on a surface apart. This raises the following interesting question: how does the energy spectrum compare to the (full) marked length spectrum? It is, at the moment, unknown to the author whether the energy spectrum carries the same information as the marked length spectrum or whether it carries strictly less information. We discuss this question in more depth in \Cref{sec:furthercomparison}.

As a corollary to our results we obtain that the set of hyperbolic metrics satisfies energy spectrum rigidity.
\begin{unnumberedcorollary}[{\Cref{cor:hyperbolicmetricenergyspecrigidity}}]
The set of hyperbolic metrics on $S$, defined up to isotopy, satisfies energy spectrum rigidity.
\end{unnumberedcorollary}
A quadratic differential on $S$ induces a singular flat metric (see \Cref{subsec:conformalgeometry}). It is proved in \cite{DuchinLeiningerRafi} that the set of these metrics satisfies simple length spectrum rigidity. It then follows from our results that this set also satisfies energy spectrum rigidity.
\begin{unnumberedcorollary}[{\Cref{cor:singflatmetricsenergyspecrigidity}}]
The set of singular flat metrics that are induced by quadratic differentials, defined up to isotopy, satisfies energy spectrum rigidity.
\end{unnumberedcorollary} 

Our interest in these questions surrounding the energy spectrum stems from the work of Labourie in \cite{LabourieCrossRatios} (as described above). He asked whether it is possible to assign to each Hitchin representation an associated point in Teichm\"uller space, in a mapping class group invariant way. In cases where the aforementioned Labourie conjecture is true such a projection can be constructed by mapping a Hitchin representation to the unique minimiser of its energy spectrum. The Labourie conjecture has been proved for real split simple Lie groups of rank two (\cite{LabourieCyclic}). Markovi{\'c} showed in a recent preprint (\cite{Markovic}) that for the semisimple Lie group $G = \Pi_{i=1}^3 \PSL(2,\R)$ the analogue of Labourie's conjecture does not hold. The conjecture, however, remains open for simple Lie groups of rank at least three. 

Considering this situation from a slightly different angle we ask ourselves how much information about a Hitchin representation is actually encoded in its energy spectrum. More concretely, we ask whether a Hitchin representation is determined, up to conjugacy, by its energy spectrum. We hope that the results of this paper are a step towards answering this question in the affirmative. We illustrate this by applying our results to the simpler setting of Kleinian surface groups. We prove the following result.

\begin{unnumberedtheorem}[{\Cref{thm:energyspectrumgiveslengthspectrumkleinian}}]
Let $\rho, \rho' \colon \Gamma \to \PSL(2,\C)$ be two Kleinian surface groups. If the energy spectra of $\rho$ and $\rho'$ coincide, then their simple simple length spectra coincide.
\end{unnumberedtheorem}

Combined with the results of Bridgeman and Canary in \cite{BridgemanCanary} we obtain the following corollary.

\begin{unnumberedcorollary}[{\Cref{cor:kleiniansurfgroupenergyspecrigidity}}]
If $\rho, \rho' \colon \Gamma \to \PSL(2,\C)$ are Kleinian surface groups with equal energy spectrum, then $\rho'$ is conjugate to either $\rho$ or $\overline{\rho}$.
\end{unnumberedcorollary}

Unfortunately, the results obtained in this paper are not enough to conclude the same for Hitchin representations. In \Cref{sec:hitchinrepresentationsenergyspectrum} we discuss briefly the further steps that would be required to do so.

\begin{acknowledgements}
The author wishes to thank Ursula Hamenst\"adt for the many useful suggestions she has made during this project and Gabriele Viaggi for fruitful discussions. The author was supported by the IMPRS graduate program of the Max Planck Institute for Mathematics.
\end{acknowledgements}
 
\section{Prerequisites}\label{sec:prerequisitesenergyspectrum}
We let $S$ be a closed and oriented surface. We will denote its genus by $g$.
\subsection{Teichm\"uller space}\label{subsec:teichmullerspace}
We recall the definition of the Teichm\"uller space of a surface. A general reference for the concepts discussed in this section is \cite{Hubbard}.

A \textit{marked complex structure} on $S$ is a pair $(X,\phi)$ where $X$ is an Riemann surface and $\phi \colon S \to X$ is an orientation preserving diffeomorphism. Two marked complex structures $(X,\phi)$ and $(X',\phi')$  are equivalent if there exists a biholomorphism $\psi \colon X' \to X$ such that $\phi^{-1}\circ \psi \circ \phi'\colon S \to S$ is isotopic to the identity map. 

\begin{definition}
The Teichm\"uller space of $S$, denoted $\TeichSpace(S)$, is the set of equivalence classes of marked complex structures on $S$.
\end{definition}

Teichm\"uller space can be equipped with a smooth structure (or even a complex structure) and if $S$ is a surface of genus $g\geq 2$, then $\TeichSpace(S)$ is diffeomorphic to $\R^{6g-6}$.

We will describe here some alternative ways to describe $\TeichSpace(S)$ which will be more practical to work with in the applications we have in mind. The complex structure on a Riemann surface $X$ is uniquely determined by an automorphism $J_X \colon TX \to TX$ that satisfies $J_X^2 = -\id$. We note that in general such an automorphism is only an almost complex structure, however on surfaces every almost complex structure is integrable and hence determines a complex structure. We see that each marking $(X,\phi)$ determines a complex structure $J = \phi^*J_X$ on $S$. It follows that we can alternatively take
\[
\TeichSpace(S) = \{J \mid J \colon TS\to TS \text{ is complex structure on }S\}/\sim
\]
as definition of Teichm\"uller space. Here we define that $J\sim J'$ if and only if a diffeomorphism $\psi \colon S\to S$ isotopic to the identity exists such that $J'= \psi^* J$. 
Furthermore, on a surface a complex structure is uniquely determined by a conformal class of metrics and vice versa. So we could also describe $\TeichSpace(S)$ as the set of conformal structures up to isotopy. Finally, if $S$ is a surface of genus at least two, then in each conformal class of metrics on $S$ there exists a unique hyperbolic metric. So we can also take
\[
\TeichSpace(S) = \{\rho \mid \rho \text{ is a hyperbolic metric on }S\}/\sim
\]
where $\rho \sim \rho'$ if $\rho' = \psi^*\rho$ for some diffeomorphism $\psi$ of $S$ that is isotopic to the identity.

The different views on Teichm\"uller space will be useful at different points in our discussion. If we consider a point $X\in \TeichSpace(S)$ we will think of this as the surface $S$ equipped with either a complex structure or a hyperbolic metric, each determined up to isotopy.
\subsection{Length of curves}
Let $\rho$ be a Riemannian metric on $S$. If $\gamma\subset S$ is a path in $S$, then we denote by $l_\rho(\gamma)$ its length measured with respect to $\rho$. If $[\gamma]$ is a free homotopy class of closed loops on $S$, then we denote
\[
\ell_\rho([\gamma]) := \inf_{\gamma'\in [\gamma]} l_\rho(\gamma').
\]
Often we will not distinguish between a closed loop on $S$ and the free homotopy class it determines and simply write $\ell_\rho(\gamma)$ for $\ell_\rho([\gamma])$. 

We will denote by $\mathcal{C}$ the set of homotopy classes of closed curves on $S$ and by $\mathcal{S} \subset \mathcal{C}$ the set of homotopy classes of simple closed curves. The \textit{marked length spectrum} of a metric $\rho$ is the vector
\[
(\ell_\rho(\gamma))_{\gamma\in \mathcal{C}} \in (\R_{>0})^\mathcal{C}.
\]
Similarly, the \textit{(marked) simple length spectrum} of a metric $\rho$ is
\[
(\ell_\rho(\gamma))_{\gamma\in \mathcal{S}} \in (\R_{>0})^\mathcal{S}.
\]
If $\mathcal{M}$ is a set of metrics on $S$, defined up to isometry, then we can ask whether the marked length spectrum or even the simple length spectrum distinguishes metrics in that set. If $\rho \mapsto (\ell_\rho(\gamma))_{\gamma\in \mathcal{C}}$ is an injection of $\mathcal{M}$ into $(\R_{>0})^\mathcal{C}$, then we say $\mathcal{M}$ satisfies \textit{length spectrum rigidity}. If the map $\rho \mapsto (\ell_\rho(\gamma))_{\gamma\in \mathcal{S}}$ injects $\mathcal{M}$ into $(\R_{>0})^\mathcal{S}$, then we say $\mathcal{M}$ satisfies \textit{simple length spectrum rigidity}.

If $[\gamma], [\eta]$ are conjugacy classes of simple closed curves on $S$, then we define their \textit{intersection number} as
\[
i([\gamma],[\eta]) = \min\{\abs{\gamma' \cap \eta'} \mid \gamma'\in [\gamma], \eta'\in [\eta]\}.
\]
If $\gamma$ and $\eta$ are simple closed curves, then, for convenience, we will write $i(\gamma,\eta)$ rather than $i([\gamma], [\eta])$. When $\gamma$ and $\eta$ are simple closed geodesics for a non-positively curved metric on $S$, then $\abs{\gamma\cap\eta}$ realises $i(\gamma,\eta)$.

\subsection{Dehn twists}
Assume $S$ has genus at least one and let $\gamma\subset S$ a simple closed curve. Let $N \subset S$ be a closed collar neighbourhood of $\gamma$ which we will identify, in an orientation preserving way, with $[0,1]\times \R/\Z$. The \textit{Dehn twist} around $\gamma$ is the orientation preserving homeomorphism $T_\gamma$ of $S$ that is equal to the identity map outside of $N$ and is given by 
\[
(t,[\theta]) \mapsto (t, [\theta + t])
\]
on $N \cong [0,1]\times \R/\Z$. Since these definitions coincide on the boundary of $N$, we see that $T_\gamma$ is indeed continuous. Note that the isotopy class of $T_\gamma$ is independent of the choice of representative in $[\gamma]$ and of the choice of collar neighbourhood $N$. In general we will refer to any homeomorphism in the isotopy class determined by $T_\gamma$ as a Dehn twist around $\gamma$. By a slight modification to the above construction it is possible to find a smooth representative of the isotopy class.

A Dehn twist defines a mapping on Teichm\"uller space. Namely, if $[(X,\phi)]\in \TeichSpace(S)$, then $T_\gamma\cdot [(X,\phi)] = [(X,\phi \circ T_\gamma^{-1})]$. To put this in a slightly broader context we note that the Dehn twist is an element of the mapping class group of the surface $S$. The mapping class group has a natural action on Teichm\"uller space which is given by precisely the mapping defined here for the Dehn twist.

If $\eta \subset S$ is a closed loop (resp. a homotopy class of closed loops), then we define $T_\gamma \eta$ to be the loop $T_\gamma \circ \eta$ (resp. the homotopy class containing this loop).

In our proof of \Cref{thm:energyspectrumgiveslengthspectrum} we will need a lower bound on the length of a loop that has been Dehn twisted often. The following lemma provides such an estimate.

\begin{lemma}\label{lem:dehntwistlengthestimate}
Let $(S,\rho)$ be an oriented surface of genus at least two equipped with a metric of non-positive curvature. For every pair $\gamma,\eta\subset S$ of simple closed curves there exists a constant $C = C(\gamma,\eta)>0$ such that
\[
\ell_\rho(T^n_\gamma\eta) \geq n \cdot i(\gamma,\eta) \cdot \ell_\rho(\gamma) - C
\]
for all $n\geq 1$.
\end{lemma}
Let $M = \widetilde{S}$ be the universal cover of $S$ equipped with the pullback metric. In our proof of \Cref{lem:dehntwistlengthestimate} we will use that $M$ is non-positively curved, both in a local sense and in a global sense. We will use \cite{BridsonHeafliger} as our reference for the facts on metric spaces of non-positive curvature that we will need. Because $\rho$ is a metric of non-positive curvature it follows that $M$ is a CAT(0) space (\cite[Section II.1]{BridsonHeafliger}). Moreover it is also a Gromov $\delta$-hyperbolic space (\cite[Section III.H.1]{BridsonHeafliger}) for some $\delta>0$ because, by the \v{S}varc-Milnor lemma, it is quasi-isometric to the Cayley graph of $\pi_1(S)$.

We first prove two auxiliary lemmas. For any two points $x,y\in M$ let us denote by $[x,y]$ the (directed) geodesic segment connecting $x$ to $y$. Furthermore, for $x,y,z\in M$ we denote by $\angle_z(x,y)$ the angle the geodesic segments $[x,z]$ and $[z,y]$ make at $z$.
\begin{lemma}\label{lem:obtuseangleestimate}
For all $x,y,z\in M$ with $\angle_z(x,y) \geq \pi/2$ we have
\[
d(x,y) \geq d(x,z) + d(y,z) - 4\delta.
\]
\end{lemma}
\begin{proof}
Because $M$ is Gromov $\delta$-hyperbolic, it follows that the triangle with vertices $x,y,z$ is $\delta$-thin (see \cite[Definition III.1.16]{BridsonHeafliger}) and hence there exist points $w_{x,y}\in [x,y], w_{x,z}\in [x,z], w_{y,z}\in [y,z]$ such that $\text{diam}(\{w_{x,y}, w_{x,z}, w_{y,z}\}) \leq \delta$. We compare the triangle with vertices $w_{x,z}, w_{y,z}, z$ to a triangle in the Euclidean plane with vertices $a,b,c$ that satisfy $d(a,c) = d(w_{x,z}, z), d(b,c) = d(w_{y,z}, z)$ and $\angle_c(a,b) = \angle_z(w_{x,z},w_{y,z})=\angle_z(x,y)\geq \pi/2$. From the CAT(0) condition follows (see \cite[Proposition II.1.7(5)]{BridsonHeafliger}) that
\[
\delta \geq d(w_{x,z}, w_{y,z}) \geq d(a,b) \geq \sqrt{d^2(w_{x,z},z)+d^2(w_{y,z}, z)}.
\]
From this we conclude that that $d(z, w_{x,z}) \leq \delta$. The triangle inequality then yields that
\[
d(w_{x,y}, z) \leq d(w_{x,y}, w_{x,z}) + d(w_{x,z}, z) \leq 2\delta.
\]
Using again the triangle inequality now gives
\begin{align*}
d(x,y) &= d(x,w_{x,y}) + d(w_{x,y},y) \geq d(x,z)-d(w_{x,y}, z) + d(y,z) - d(w_{x,y},z)\\
&\geq d(x,z) + d(y,z) - 4\delta.
\end{align*}
\end{proof}
Consider three points $x,y,z\in M$ and let $\gamma_{x,y}\colon [0,1]\to M$ be a parametrization of $[x,y]$ with $\gamma_{x,y}(0) = x$ and $\gamma_{x,y}(1) = y$. Similarly let $\gamma_{y,z}$ be a parametrization of $[y,z]$. We say the angle that $[x,y]$ and $[y,z]$ make at $y$ is positively oriented if $(\dot{\gamma}_{x,y}(1), \dot{\gamma}_{y,z}(0))$ is a positively oriented frame of $T_{y}M$ (recall that $S$ is oriented and hence also $M$). We say it is negatively oriented otherwise.

Consider a continuous path consisting of a concatination of geodesic segments $[x_0, x_1], [x_1, x_2], \hdots, [x_{n-1}, x_n]$ with pairwise distinct points $x_i\in M$. We call such a path a \textit{stairstep path} if all successive segments meet each other orthogonally and the orientation of the angle between segments at points $x_i$ is alternately positive and negative. So either each angle at even numbered points is positively oriented and negatively oriented at odd numbered points or it is the other way around.
\begin{lemma}\label{lem:lengthestimatestairsteppath}
If the segments $[x_0, x_1], [x_1, x_2], \hdots, [x_{n-1}, x_n]$ form a stairstep path, then
\[
d(x_0,x_n) \geq \sum_{i = 0}^n d(x_i, x_{i+1}) - 4(n-1)\delta.
\]
\end{lemma}
\begin{proof}
For $i=0,\hdots, n-1$ let $L_i$ be the geodesic in $M$ that contains the segment $[x_i,x_{i+1}]$. A pair of geodesics $L_i$, $L_{i+2}$ is connected by a segment $[x_{i+1}, x_{i+2}]$ that meets both geodesics orthogonally. It follows from convexity of the distance function that this is the unique geodesic segment that realises the shortest path between $L_i$ and $L_{i+2}$. Because we assumed that the points $x_i$ are pairwise distinct it follows that $L_i$ and $L_{i+2}$ are a positive distance apart. In particular, they do not intersect.

Each $L_i$ divides the manifold $M$ into two halves. For $i=0,\hdots, n-2$ let $H_i$ be the component of $M-L_i$ that contains $x_n$. From the assumption that successive angles have opposite orientation it follows that $x_n$ and $x_{n-3}$ lie on opposite sides of $L_{n-2}$ and hence $x_{n-3}\not\in H_{n-2}$. Because the segment $[x_{n-4}, x_{n-3}]$ is contained in $L_{n-2}$ which is disjoint from $L_{n-2}$, we also have $x_{n-4}\not\in H_{n-2}$. We claim the same holds for $x_{n-5}$. Since $L_{n-4}$ and $L_{n-2}$ do not intersect, it follows that $L_{n-2}\cup H_{n-2} \subset H_{n-4}$. Note that $x_{n-2}\in L_{n-2} \subset H_{n-4}$. Using again the assumption that successive angles have opposite orientation we find that $x_{n-2}$ and $x_{n-5}$ lie on opposite sides of $L_{n-4}$, hence we must have $x_{n-5}\not\in H_{n-4}$. Because $H_{n-2}\subset H_{n-4}$ we conclude that in particular $x_{n-5}\not\in H_{n-2}$. Continuing this argument inductively we find that $x_0\not\in H_{n-2}$ or, in other words, $x_0$ and $x_n$ lie on opposite sides of $L_{n-2}$.

We now prove the lemma by induction on $n$, the number of segments. For $n=1$ the statement is trivial and for $n=2$ it follows directly from \Cref{lem:obtuseangleestimate}. Assume the lemma holds for some $n\geq 2$. Consider a stairstep path $[x_0,x_1],\hdots, [x_n, x_{n+1}]$ consisting of $n+1$ segments. Let $L_{n-1}$ as defined above. Then the segments $[x_0,x_n]$ and $[x_n,x_{n+1}]$ lie on opposite sides of $L_{n-1}$ and meet at $x_n\in L_{n-1}$. Because the segment $[x_n, x_{n-1}]$ is orthogonal to $L_{n-1}$, it follows that $\angle_{x_n}(x_0, x_{n+1})\geq \pi/2$. We apply \Cref{lem:obtuseangleestimate} to find
\begin{align*}
d(x_0, x_{n+1}) &\geq d(x_0, x_n) + d(x_n, x_{n+1}) - 4\delta\\&\geq \sum_{i = 0}^n d(x_i, x_{i+1}) + d(x_n, x_{n+1}) - 4(n-1)\delta - 4\delta\\
&= \sum_{i=0}^{n+1} d(x_i,x_{i+1}) - 4n \delta.
\end{align*}
Here the second inequality follows from the induction assumption. We see that the lemma also holds for paths consisting of $n+1$ segments. This concludes the argument.
\end{proof}

\begin{proof}[Proof of \Cref{lem:dehntwistlengthestimate}.]
The statement is trivial if $i(\gamma,\eta) = 0$. Hence from now on we assume that $i(\gamma,\eta)>0$. Take $\gamma$ and $\eta$ to be geodesic representatives in $(S,\rho)$ of their free homotopy class. These loops realise the minimal number of intersections so $k := i(\gamma,\eta) = \abs{\gamma\cap \eta}$. We label the intersection points $\gamma\cap \eta = \{p_1,\hdots, p_k\}$ in order of appearance along some parametrization of $\eta$. Cut $\eta$ into $k$ pieces $\eta^1,\hdots, \eta^k$, where each $\eta^i$ is the subarc connecting $p_i$ to $p_{i+1}$ (and $\eta^k$ connects $p_k$ to $p_1$).

For each $i=1,\hdots, k$ let $A_i$ be the geodesic arc of minimal length in the homotopy class of $\eta^i$ with endpoints sliding freely over $\gamma$. Each arc $A_i$ meets $\gamma$ orthogonally because it is length minimizing. The loop $\eta$ is homotopic to a unique loop $\omega_0$ consisting of a concatination of geodesic arcs 
\[
A_1, B_{1,0}, A_2, B_{2,0}, \hdots, A_{k}, B_{k,0}
\]
where each $B_{i,0}$ is an arc that lies along the geodesic $\gamma$. Similarly, the Dehn twisted loops $T^n_\gamma\eta$ are homotopic to a unique loop $\omega_n$ consisting of segments $A_1, B_{1,n}, \hdots, A_{k}, B_{k,n}$. Each $B_{i,n}$ differs from $B_{i,0}$ by $n$ turns around $\gamma$.

After untwisting any turns that $\eta$ made around $\gamma$ in the opposite direction of the Dehn twist we find that for $n$ high enough the angle between each $A_i$ and $B_{i,n}$ is positively oriented and the angle between each $B_{i,n}$ and $A_{i+1}$ is negatively oriented. It follows that if we lift $\omega_n$ to $M$ it is a stairstep path. We also see there exists a constant $c>0$ such that $l_\rho(B_{i,n}) \geq n\cdot \ell_\rho(\gamma) - c$ for all $i=1,\hdots, k$ and $n\geq 1$.

Consider the geodesic representatives $\eta_n$ of the homotopy classes $T^n_\gamma\eta$. Because for $n$ high enough the arc $B_{1,n}$ winds around $\gamma$ at least once, it follows that $\eta_n$ and $\omega_n$ intersect at least once. Parametrize $\eta_n \colon [0,1]\to S$ to start at such an intersection point and consider a lift $\widetilde{\eta}_n$ to $M$. The endpoints of $\widetilde{\eta}_n$ are connected by the stairstep path that is a lift of $\omega_n$. We use \Cref{lem:lengthestimatestairsteppath} to conclude that
\begin{align*}
\ell_\rho(\eta) &= d(\widetilde{\eta}_n(0), \widetilde{\eta}_n(1)) \geq \sum_{i=0}^k (l_\rho(A_i) + l_\rho(B_i)) - 4k \delta\\&\geq n \cdot k \cdot \ell_\rho(\gamma) - (4\delta + c)\cdot k
\\&= n \cdot i(\gamma,\eta) \cdot \ell_\rho(\gamma) - C
\end{align*}
where we take $C= (4\delta+c)\cdot k$.
\end{proof}

\subsection{Conformal geometry of surfaces}\label{subsec:conformalgeometry}

In this section we will consider some of the conformal aspects of the geometry of a closed surface. We let $X$ be a closed Riemann surface.

\begin{definition}
Let $\gamma \subset X$ be a closed curve. We define the extremal length of $\gamma$ in $X$ to be
\begin{align}\label{eq:analyticdefextremallength}
E_X(\gamma) = \sup_{\sigma} \frac{\ell_\sigma^2(\gamma)}{\Area(\sigma)}.
\end{align}
Here the supremum runs over all metrics in the conformal class determined by $X$.
\end{definition}
In case $\gamma$ is a simple closed curve a second equivalent definition for its extremal length exists. We will denote the modulus of an annulus $A\subset X$ by $M(A)$.
\begin{definition}
If $\gamma\subset X$ is a simple closed curve, then
\begin{align}\label{eq:geometricdefextremallength}
E_X(\gamma) = \inf_{A} \frac{1}{M(A)}
\end{align}
where the infimum runs over all annuli in $X$ whose core curve is homotopic to $\gamma$.
\end{definition}
When $\gamma$ is a simple closed curve, then the metric realising the supremum in \Cref{eq:analyticdefextremallength} and the annulus realising the infimum in \Cref{eq:geometricdefextremallength} can be explicitly described. In order to do this we need to consider Strebel differentials on $X$ which we will describe here. We refer to \cite{Strebel} as a reference on Strebel differentials and quadratic differentials in general.

A \textit{quadratic differential} $\phi$ on $X$ is a differential that in any local coordinates can be written as $\phi = \phi(z) dz^2$ with $\phi(z)$ a holomorphic function. A quadratic differential determines two singular foliations of $X$. Namely, away from the zeroes of $\phi$, lines that have tangent directions $v\in TS$ with $\phi(v,v)>0$ form a foliation called the \textit{horizontal foliation} of $\phi$ and lines with $\phi(v,v)<0$ form its \textit{vertical foliation}. The leaves of these foliations are called \textit{singular} if they terminate in a zero of $\phi$ and are called \textit{non-singular} otherwise. Furthermore, a quadratic differential also determines a flat singular metric on $S$ which can be expressed as $\abs{\phi(z)}\abs{dz}^2$ in local coordinates. Around any point on $S$ that is not a zero of $\phi$ there exist complex coordinates in which $\phi = dz^2$. In these coordinates the singular flat metric is simply the Euclidean metric $\abs{dz}^2$, the horizontal foliation consists of the lines with constant $\Imag z$ and the vertical foliation consists of the lines with constant $\Real z$.

For every simple closed curve $\gamma\subset X$ there exists a unique quadratic differential, called the \textit{Strebel differential}, such that every non-singular leaf of the horizontal foliation of the differential is closed and homotopic to $\gamma$. The annulus obtained by taking the union of these non-singular leaves realises the infimum in \Cref{eq:geometricdefextremallength}. The singular flat metric that is determined by the Strebel differential realises the supremum in \Cref{eq:analyticdefextremallength}.

We will prove here some results on the extremal length of intersecting curves that we will need in our proofs below.

\begin{lemma}
Let $\gamma,\eta \subset X$ be simple closed curves. Then
\begin{align}\label{eq:collarlemma}
E_X(\gamma) E_X(\eta) \geq i(\gamma,\eta)^2.
\end{align}
\end{lemma}

\begin{proof}
Consider the Strebel differential of $\gamma$ on $X$. Let $A\subset X$ be the annulus consisting of the union of all non-singular leaves of its horizontal foliation. Then we have $M := M(A) = 1/E_X(\gamma)$. Consider on $X$ the singular flat metric $\sigma$ determined by the Strebel differential. Normalise such that the annulus $A$ has circumference 1 and height $M$. Any curve homotopic to $\eta$ crosses the annulus at least $i(\gamma,\eta)$ times and hence $\ell_\sigma(\eta) \geq i(\gamma,\eta) \cdot M$. Then from  \Cref{eq:analyticdefextremallength} we see that
\[
E_X(\eta) \geq \frac{\ell_\sigma(\eta)^2}{\Area(\sigma)} \geq \frac{i(\gamma,\eta)^2 M^2}{M} = \frac{1}{E_X(\gamma)} i(\gamma,\eta)^2.
\]
This proves the result.
\end{proof}

\begin{lemma}\label{lem:almostcollarequality}
Let $S$ be a surface of genus at least two and let $\gamma\subset S$ a simple closed curve. Then there exists a simple closed curve $\eta \subset S$, satisfying $i(\gamma,\eta) \in \{1, 2\}$, such that for every $\epsilon>0$ there exists a complex structure $X$ on $S$ with
\[
E_X(\gamma) E_X(\eta) \leq i(\gamma,\eta)^2 + \epsilon
\]
and
\[
1-\epsilon \leq E_X(\gamma) \leq 1+\epsilon.
\]
\end{lemma}
\begin{proof}
We construct the complex structure on $S$ by cutting and pasting together several pieces. The main idea is to start with a smaller Riemann surface and curves $\gamma,\eta$ for which \Cref{eq:collarlemma} is an equality. Then we add pieces to this surface to make it of the same topological type as $S$ in a way that does not disturb the quantity $E_X(\gamma)E_X(\eta)$ to much.

For our construction we need to distinguish between two cases, namely whether $\gamma$ is a separating curve or not. We will start with the case that $\gamma$ is separating which is the more complicated case. The curve $\gamma$ separates $S$ into two surfaces $S', S''$ with border. Denote by $g', g''\geq 1$ their respective genus. Then the genus of $S$ equals $g=g'+g''$.

We start by considering a square with side lengths 1 in $\C$. We glue the boundary according to the gluing pattern given in \Cref{fig:gluingpattern} to obtain the 2-sphere. We denote by $X_0$ the 2-sphere equipped with the complex structure determined by this gluing. We consider two simple closed loops $\gamma'$ and $\eta'$ on the sphere as specified in \Cref{fig:gluingpattern}. Fix a small constant $\delta>0$. In each of the four components of the complement of $\gamma'\cup \eta'$ we cut a slit of length $\delta$ at the locations as indicated in \Cref{fig:gluingpattern} (the slits are marked by (I) through (IV)). We let $X'$ be an arbitrary closed Riemann surface of genus $g'-1$. At arbitrary points in $X'$ we cut two slits. We glue one of these slits to the slit marked (I) in $X_0$. The other slit we glue to the slit marked (II). Similarly, we take $X''$ an arbitrary Riemann surface of genus $g''-1$, again cut two slits and glue $X''$ to $X_0$ by gluing one of these slits to the slit marked (III) and the other to the slit marked (IV).

\begin{figure}[ht]
\centering
{
	\resizebox{70mm}{!}
	{\Large
		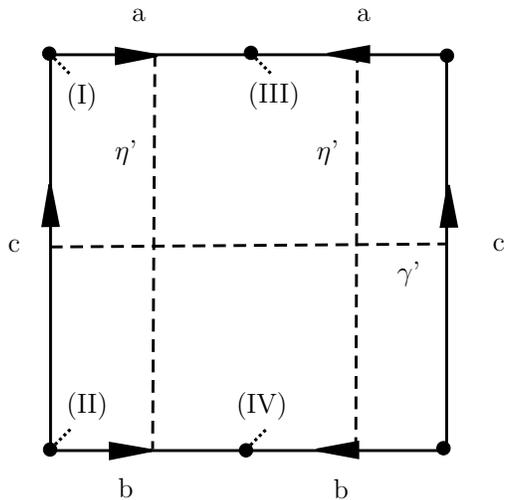
	}
}
\caption{A gluing pattern on the boundary of a square. Edges labelled with the same letter are glued together according to the orientation indicated by the arrows. We cut slits of length $\delta$ at the places indicated by (I) through (IV).}\label{fig:gluingpattern}
\end{figure}

We denote by $X = X_0 \sqcup X' \sqcup X'' /\sim$ the Riemann surface that is obtained from these gluings. Let us first make the observation that in $X$ the curves $\gamma'$ and $\eta'$ are no longer null homotopic (as they were on the sphere) and they satisfy $i(\gamma',\eta') = 2$. Secondly, we note that the genus of $X$ equals $g$. Namely, the combined genus of $X'$ and $X''$ contributes $g'+g''-2$ to the genus of $X$ and the fact that we glued each surface along two slits contributes $2$ more (see \Cref{fig:gluingoverview}).

\begin{figure}[ht]
\centering
{
	\resizebox{120mm}{!}
	{\huge
		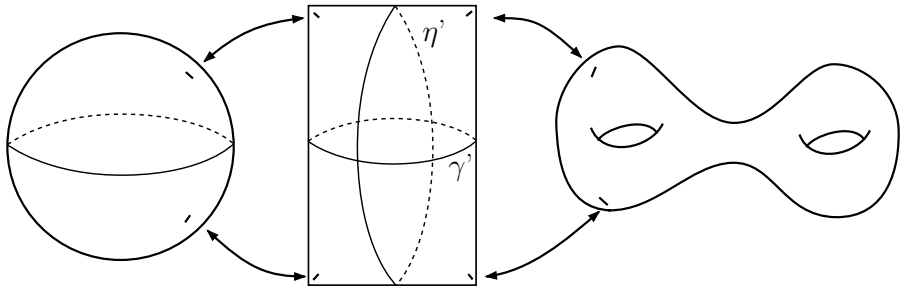
	}
}
\caption{Example of a gluing as described above with $g'-1=0$ and $g''-1=2$.}\label{fig:gluingoverview}
\end{figure}

Consider the square in $\C$ from which we glue $X_0$. We note that the $1/2-\delta$ neighbourhood of the curve $\gamma'$ in the square intersects no slits. This neighbourhood descends to an annulus in $X$ around $\gamma'$ that has modulus $1-2\delta$. From \Cref{eq:geometricdefextremallength} it follows that $E_X(\gamma') \leq 1/(1-2\delta)$. Similarly the $1/4-\delta$ neighbourhood of $\eta'$ in the square intersects no slits and descends to an annulus in $X$ around $\eta'$. Its modulus equals $1/4-\delta$ and hence $E_X(\eta') \leq 1/(1/4-\delta)$. We now see that for any $\epsilon>0$ there is a $\delta$ small enough such that
\[
E_X(\gamma')E_X(\eta') \leq \frac{1}{1-2\delta} \cdot \frac{1}{1/4-\delta} \leq 4 + \epsilon = i(\gamma,\eta)^2 + \epsilon
\]
and $E_X(\gamma') \leq 1+\epsilon$. For the lower bound on $E_X(\gamma')$ we combine \Cref{eq:collarlemma} with $E_X(\eta') \leq 1/(1/4-\delta)$ to find that also $E_X(\gamma')\geq 1-\epsilon$ for $\delta$ small enough.

Finally we note that $\gamma'$ separates $X$ into two surfaces with border that have genus $g'$ and $g''$ respectively. It follows from the classification of surfaces that these two subsurfaces are diffeomorphic to the two corresponding subsurfaces of $S$. By gluing these diffeomorphisms together we find that there exists a diffeomorpism between $X$ and $S$ that sends the homotopy class of $\gamma'$ to that of $\gamma$. We let $\eta$ be the simple closed curve in $S$ that corresponds to $\eta'$ under this diffeomorphism. We note that the homotopy class of $\eta$ only depends on the placement of the slits in $X_0$ we glued along and not on the constant $\delta$. Hence we can take $\eta$ the same for all choices of $\epsilon$. Using this diffeomorphism we equip $S$ with a complex structure that satisfies the bounds on the extremal length of $\gamma$ and $\eta$.

The case where $\gamma$ is non-separating is easier. In this case we take $X_0$ to be a torus and $\gamma'$ and $\eta'$ a pair of simple closed curves with $i(\gamma',\eta') = 1$. By picking a suitable complex structure on the torus we can realise equality in \Cref{eq:collarlemma} and $E_X(\gamma')=1$. We glue an arbitrary Riemann surface of genus $g-1$ to the torus along a single small slit to obtain a Riemann surface $X$ of genus $g$. Again by the classification of surfaces we can find a diffeomorphism between $X$ and $S$ that takes $\gamma'$ to $\gamma$. The estimate on the extremal lengths $\gamma$ and $\eta$ in this case is similar to the previous case.
\end{proof}

\subsection{Harmonic maps}\label{sec:harmonicmaps}
Let $(M,\sigma)$ and $(N,\rho)$ be Riemannian manifolds. Consider a Lipschitz continuous map $f \colon N \to N$. We define its \textit{energy density} $e(f) \colon M \to \R$ to be
\[
e(f) = \frac{1}{2} \norm{df}^2
\]
where the norm $\norm{\cdot}$ is the Hilbert-Schmidt norm on the vector bundle $T^*M\otimes f^*TN$ induced by the metrics $\sigma$ and $\rho$. The energy density is a pointwise measure of the amount of stretching that a map does. We note that as $f$ is Lipschitz continuous it is differentiable almost everywhere and hence $e(f)$ is defined almost everywhere. The \textit{Dirichlet energy} of $f$ is defined as
\[
\mathcal{E}(f) = \int_M e(f) \vol_\sigma.
\]
A critical point of this energy functional is called a \textit{harmonic map}. If $\sigma$ and $\rho$ are smooth Riemannian metrics, then a harmonic map is also smooth. 

A straightforward calculation shows that if $M$ is a surface, then the Dirichlet energy of a map is independent of conformal scalings of the metric $\sigma$. It follows that in this case the harmonicity of a map and its energy depend only on the conformal structure on the surface. If we want to stress the dependence of the energy on a complex structure $J$ on $M$ and the metric $\rho$ on $N$ we will write $e(f;J,\rho)$ for the energy density and $\mathcal{E}(f;J,\rho)$ for the Dirichlet energy of a map $f$.

We will make use of the following lemma by Minsky.
\begin{lemma}[{\cite[Proposition 3.1]{MinskyHarmonic}}]\label{lem:minskyestimate}
Let $X$ be a Riemann surface and $(N,\rho)$ be a Riemannian manifold. For any map $f \colon X \to (N,\rho)$ and any simple closed curve $\gamma\subset X$ we have
\[
\mathcal{E}(f) \geq \frac{1}{2} \frac{\ell_\rho^2(f \circ \gamma)}{E_X(\gamma)}.
\]
\end{lemma}

\section{The energy spectrum}\label{sec:theenergyspectrum}
In this section we introduce the energy spectrum of a Riemannian manifold and study its relation to the simple length spectrum. 

Let $S$ be a surface of genus at least two and let $(N,\rho)$ be a Riemannian manifold. We fix a homotopy class $[f]\in [S,N]$ of maps from $S$ to $N$. For every complex structure $J$ on $S$ we consider the quantity
\[
\Energy(J) = \inf_{h \in [f]} \mathcal{E}(h; J, \rho).
\]
Here the infimum is taken over all Lipschitz continuous maps in the homotopy class $[f]$. If $\phi \colon S \to S$ is a diffeomorphism, then $\phi \colon (S,\phi^*J) \to (S,J)$ is a biholomorphism. In particular we have $\mathcal{E}(h\circ \phi;\phi^*J,\rho) = \mathcal{E}(h;J,\rho)$. It follows that if $\phi$ is isotopic to the identity, then $\Energy(J) = \Energy(J\circ\phi)$ and we see that the function $\Energy$ descends to a well-defined function on Teichm\"uller space.
\begin{definition}\label{def:energyspectrum}
The energy spectrum of $(N,\rho)$ and $[f]$ is the function $\Energy \colon \TeichSpace(S) \to \R$ given by
\[
\Energy([J]) = \inf_{h\in [f]} \mathcal{E}(h;J,\rho)
\]
where the infimum is taken over all Lipschitz continuous maps in $[f]$.
\end{definition}
We will often suppress the dependence on a choice of the homotopy class $[f]$ and simply refer to the energy spectrum of $(N,\rho)$. 

The energy spectrum gives a rough measure of the compatibility between $(N,\rho)$ and points in Teichm\"uller space. Namely the quantity $\Energy([J])$ measures how much the complex surface $(S,J)$ must be stretched for it to be mapped into $(N,\rho)$.

\begin{proposition}
The energy spectrum $\Energy \colon \TeichSpace(S) \to \R$ is a continuous function on Teichm\"uller space.
\end{proposition}
\begin{proof}
If $\sigma$ is a Riemannian metric on $S$ and $h\colon S \to N$ a Lipschitz continuous map, then the energy density of $h$ with respect to $\sigma$ is given, at a point $x\in S$ where $h$ is differentiable, by
\begin{equation}\label{eq:energydensity}
e(f;\sigma,\rho) = \frac{1}{2}\sum_{i=1}^2 \norm{dh(e_i)}^2_\rho 
\end{equation}
where $(e_1,e_2)$ is an orthonormal basis of $T_xS$ with respect to $\sigma$. If $\sigma'$ is a second Riemannian metric, then by compactness of $S$ there exists a Lipschitz constant $C(\sigma,\sigma')\geq 1$ such that
\[
\frac{\sigma(v,v)}{C(\sigma,\sigma')} \leq \sigma'(v,v) \leq C(\sigma,\sigma')\cdot \sigma(v,v) \text{ for all } v\in TS.
\]
For any $x\in S$ we can simultaneously diagonalise the metrics at $x$ to find a basis $(e_1,e_2)$ of $T_xS$ that is orthonormal for $\sigma$ and orthogonal for $\sigma'$. If we denote $\lambda_i = \sigma'(e_i,e_i)$, then $1/C(\sigma,\sigma') \leq \lambda_i \leq C(\sigma,\sigma')$. The basis $(e_1/\sqrt{\lambda_1}, e_2/\sqrt{\lambda_2})$ is orthonormal for $\sigma'$ and from the expression of the energy density given in \Cref{eq:energydensity} now follows that
\[
\frac{e(h;\sigma',\rho)}{C(\sigma,\sigma')} \leq e(h;\sigma, \rho) \leq C(\sigma, \sigma') \cdot e(h;\sigma', \rho).
\]
By integrating we see that similar inequalities hold true for $\mathcal{E}(h;\sigma,\rho)$ and $\mathcal{E}(h;\sigma',\rho)$. Then taking the infimum over all $h\colon S \to N$ Lipschitz continuous in the homotopy class $[f]$ gives
\begin{equation}\label{eq:energycomparison}
\frac{\Energy([\sigma'])}{C(\sigma,\sigma')} \leq \Energy([\sigma]) \leq C(\sigma,\sigma')\cdot \Energy([\sigma']).
\end{equation}
Now suppose $X_n$ is a sequence in Teichm\"uller space converging to a point $X\in \TeichSpace(S)$. The points $X_n$ and $X$ can be represented by hyperbolic metrics $\sigma_n$ and $\sigma$ such that $\sigma_n \to \sigma$ uniformly on $S$ as $n\to\infty$. It follows that the Lipschitz constants can be taken such that $C(\sigma_n,\sigma) \to 1$. Then \Cref{eq:energycomparison} gives that $\Energy(X_n) \to \Energy(X)$ for $n\to\infty$ and thus $\Energy\colon \TeichSpace(S) \to \R$ is indeed a continuous function.
\end{proof}

If we assume that for every complex structure there exists an energy minimizing harmonic map $f_J \colon (S,J) \to (N,\rho)$ in the homotopy class $[f]$, then $\Energy([J]) = \mathcal{E}(f_J; J,\rho)$. By the classical results of \cite{EellsSampson} this is for example the case if $(N,\rho)$ is compact and has non-positive curvature. If the harmonic maps $f_J$ are unique and satisfy certain non-degeneracy conditions, then they depend smoothly on the complex structure (see \cite{EellsLemaire}). This happens for example if $(N,\rho)$ is negatively curved and the map $f$ can not be homotoped into the image of a closed geodesic. In this case the energy spectrum $\Energy$ is a smooth map on Teichm\"uller space.

To state our main result we will restrict to the situation where $N = S$ is a surface of genus at least two, $[f] =[\id]$ and $\rho$ is a non-positively curved Riemannian metric on $S$. 

\begin{theorem}\label{thm:energyspectrumgiveslengthspectrum}
Let $\rho, \rho'$ be non-positively curved Riemannian metrics on a surface $S$ of genus at least two. If the energy spectra of $(S,\rho)$ and $(S,\rho')$ (with $[f]=[\id]$) coincide, then the simple length spectra of $\rho$ and $\rho'$ coincide.
\end{theorem}

Simply put, the energy spectrum of a metric determines its simple length spectrum. In fact, we will detail a procedure that recovers the length of a simple closed curve from the information given by the energy spectrum. Our principal observation is that when repeatedly Dehn twisting around a simple closed curve the quadratic growth rate of the energy is proportional to the square of the length of that curve in $(S,\rho)$.

We now start our proof of \Cref{thm:energyspectrumgiveslengthspectrum}. For this we fix a non-positively curved Riemannian metric $\rho$ on $S$. We let $\Energy \colon \TeichSpace(S) \to \R$ be its energy spectrum.

\begin{definition}\label{def:taunotation}
For $\gamma\subset S$ a simple closed curve, $X \in \TeichSpace(S)$ and $n\in \N$ we define
\[
\tau(X,\gamma,n) = \frac{\Energy(T^n_\gamma X)}{n^2}
\]
and
\[
\tau^-(X,\gamma) = \liminf_{n\to\infty} \tau(X,\gamma,n) \text{ and } \tau^+(X,\gamma) = \limsup_{n\to\infty} \tau(X,\gamma,n).
\]
\end{definition}
\begin{remark}\label{rmk:energyofdehntwist}
The value of the energy spectrum at the point $T^n_\gamma X$ can alternatively be characterised as
\[
\Energy(T^n_\gamma X) = \inf_{h'\in [T^n_\gamma]} \mathcal{E}(h';J,\rho)
\]
where the infimum runs over all Lipschitz continuous maps $h'\colon S\to S$ homotopic to $T^n_\gamma$. To see this we let $J$ be a complex structure on $S$ representing $X\in \TeichSpace(S)$. Then the complex structure $(T^{-n}_\gamma)^*J$ is a representative of $T^n_\gamma X$. Now the map $T^n_\gamma \colon (S,J) \to (S,(T_\gamma^{-n})^* J)$ is a biholomorphism, hence for any Lipschitz continuous map $h \colon S \to S$ we have $\Energy(h; (T_\gamma^{-n})^* J, \rho) = \Energy(h\circ T^n_\gamma;J,\rho)$. Noting that $h\in [\id]$ if and only if $h \circ T^n_\gamma \in [T^n_\gamma]$ we find that indeed
\[
\Energy(T^n_\gamma X) = \inf_{h\in [\id]}\mathcal{E}(h; (T_\gamma^{-n})^* J, \rho) = \inf_{h'\in [T^n_\gamma]} \mathcal{E}(h';J,\rho).
\]
\end{remark}
We will now show that the quantities $\tau^-(\cdot, \gamma)$ and $\tau^+(\cdot, \gamma)$ can be used to measure $\ell_\rho(\gamma)$.
\begin{lemma}\label{lem:tauupperbound}
For any $X\in \TeichSpace(S)$ and $\gamma\subset S$ a simple closed curve we have
\[
\tau^+(X,\gamma) \leq \frac{1}{2} E_X(\gamma) \cdot \ell_\rho^2(\gamma).
\]
\end{lemma}
\begin{proof}
Consider a complex structure on $S$ that represents $X\in \TeichSpace(S)$. For convenience we will denote $S$ equipped with this choice of complex structure also as $X$. 

We will find an upper bound for the quantity $\Energy(T^n_\gamma X)$. To this end we construct a Lipschitz continuous map $k_n \colon X \to (S,\rho)$ in the homotopy class of $T^n_\gamma$ for which we have an explicit bound on its energy. Then the observations of \Cref{rmk:energyofdehntwist} will imply that $\Energy(T_\gamma^n X)\leq \mathcal{E}(k_n)$.

Consider the Strebel differential on $X$ for the curve $\gamma$. Let $A$ be the annulus in $X$ consisting of the union of all non-singular horizontal leaves of this Strebel differential. If $M = M(A)$ is the modulus of $A$, then $E_X(\gamma) = 1/ M$. By uniformising $A$ we can find a conformal identification between $A$ and the flat cylinder $[0,M]\times \R/\Z$. We use this to equip $A$ with coordinates $(x,[y])\in [0,M]\times \R/\Z$.

Let $\eta \colon \R/\Z \to (S,\rho)$ be a length minimising geodesic loop freely homotopic to $\gamma$ (so $\ell_\rho(\gamma) = l_\rho(\eta)$). Let $0 < \epsilon < M/2$ arbitrary. By deforming the identity map of $S$ we can find a Lipschitz continuous map $k_0 \colon X \to S$ that is homotopic to the identity and on the subcylinder 
\[
A_\epsilon = \{(x,[y]) \mid \epsilon \leq x \leq M-\epsilon\}
\]
is given by $k_0(x,[y]) = \eta([y])$. Let $Y$ be the complement of $A_\epsilon$ in $X$. We set $C = \mathcal{E}(k_0\vert_{Y})$ which is a constant depending only on our choice of $k_0$ (which in turn depends only on $\epsilon$).

For $n\in \N$ we define the maps $k_n \colon X \to S$ as follows. On $Y$ we set $k_n \vert_{Y} \equiv k_0\vert_{Y}$ and on $A_\epsilon$ we put
\[
k_n(x,[y]) = \eta\left(\left[y + n \cdot \frac{x-\epsilon}{M-2\epsilon}\right]\right).
\]
The map $k_n$ coincides with $k_0$ on the boundaries of $A_\epsilon$ and hence each $k_n$ defines a Lipschitz continuous map on $X$. Note that each $k_n$ is homotopic to $T_\gamma^n$.

We now calculate the energy of the maps $k_n$. To this end this we equip $A_\epsilon$ with the conformal flat metric obtained from the identification $A \cong [0,M]\times\R/\Z$. Using this choice of metric, we find on $A_\epsilon$ that
\[
e(k_n) = \frac{1}{2}\left\lbrace \norm*{\parder{k_n}{x}}^2 + \norm*{\parder{k_n}{y}}^2 \right\rbrace = \frac{1}{2} \left\lbrace \left(\frac{n}{M-2\epsilon} \right)^2 + 1 \right\rbrace \norm{\dot{\eta}}^2.
\]
Hence
\begin{align*}
\mathcal{E}(k_n\vert_{A_\epsilon}) &= \int_0^1 \int_\epsilon^{M-\epsilon} e(k_n) dx dy \\&
=  \frac{1}{2} \left\lbrace \left(\frac{n}{M-2\epsilon} \right)^2 + 1 \right\rbrace \cdot \int_0^1 \int_\epsilon^{M-\epsilon} \norm{\dot{\eta}}^2 dx dy \\&
=  \frac{1}{2} \left\lbrace \left(\frac{n}{M-2\epsilon} \right)^2 + 1 \right\rbrace \cdot (M-2\epsilon) \cdot \ell_\rho^2(\gamma).
\end{align*}
We can now estimate (cf. \Cref{rmk:energyofdehntwist})
\begin{align*}
\tau(X,\gamma,n) &= \Energy(T^n_\gamma X) \leq \mathcal{E}(k_n) = \mathcal{E}(k_n\vert_{A_\epsilon}) + \mathcal{E}(k_n\vert_Y) \\&= \frac{1}{2} \left\lbrace
\frac{n^2}{M-2\epsilon} + M-2\epsilon \right\rbrace\cdot  \ell_\rho^2(\gamma) + C.
\end{align*}
By dividing by $n^2$ and taking the limit superior for $n\to\infty$ we find
\[
\tau^+(X,\gamma) \leq \frac{1}{2} \frac{1}{M-2\epsilon} \cdot \ell_\rho^2(\gamma).
\]
Finally noting that $\epsilon>0$ was arbitrary we conclude that
\[
\tau^+(X,\gamma) \leq \frac{1}{2} \frac{1}{M} \cdot \ell_\rho^2(\gamma)  = \frac{1}{2} E_X(\gamma) \cdot \ell_\rho^2(\gamma).
\]
\end{proof}

\begin{lemma}\label{lem:taulowerbound}
For any $X \in \TeichSpace(S)$ and simple closed curves $\gamma,\eta \subset S$ we have
\[
\tau^-(X,\gamma) \geq \frac{1}{2} \frac{i(\gamma,\eta)^2 \cdot \ell_\rho^2(\gamma)}{E_X(\eta)}
\]
\end{lemma}
\begin{proof}
Let us again, by abuse of notation, denote by $X$ both a point in Teichm\"uller space and a Riemann surface representing it. The lemma follows easily from \Cref{lem:minskyestimate} and \Cref{lem:dehntwistlengthestimate}. Namely, from the latter follows that a constant $C =C(\gamma,\eta)>0$ exists such that
\[
\ell_\rho(T^n_\gamma \eta) \geq n \cdot i(\gamma,\eta) \cdot \ell_\rho(\gamma) - C.
\]
Any map $h \colon X \to (S,\rho)$ homotopic to $T^n_\gamma$ maps $\eta$ to a curve homotopic to $T^n_\gamma\eta$. Now \Cref{lem:minskyestimate} gives a lower bound on the energy of such maps. It follows that
\[
\tau(X,\gamma,n) = \Energy(T_\gamma^nX) \geq \frac{1}{2} \frac{(n\cdot i(\gamma,\eta) \cdot \ell_\rho(\gamma) - C)^2}{E_X(\eta)}.
\]
Dividing by $n^2$ and taking the limit inferior for $n\to\infty$ gives
\[
\tau^-(X,\gamma) \geq \frac{1}{2} \frac{i(\gamma,\eta)^2 \cdot \ell_\rho^2(\gamma)}{E_X(\eta)}.
\]
\end{proof}
We now have for any $X\in \TeichSpace(S)$ and $\gamma,\eta \subset S$ simple closed curves that
\begin{align}\label{eq:lengthestimatefromenergy}
\frac{1}{2} \frac{i(\gamma,\eta)^2 \cdot \ell_\rho^2(\gamma)}{E_X(\eta)} \leq \tau^-(X,\gamma) \leq \tau^+(X,\gamma) \leq \frac{1}{2} E_X(\gamma) \cdot \ell_\rho^2(\gamma).
\end{align}
We observe that these bounds are close together if the quantity $E_X(\gamma)E_X(\eta)$ is close to $i(\gamma,\eta)^2$. We use \Cref{lem:almostcollarequality} to finish the proof of \Cref{thm:energyspectrumgiveslengthspectrum}.

\begin{proof}[Proof of \Cref{thm:energyspectrumgiveslengthspectrum}]
Fix a simple closed curve $\gamma\subset S$. We invoke \Cref{lem:almostcollarequality} to find a simple closed curve $\eta\subset S$ with $i(\gamma,\eta)>0$ and for every $k\in \N$ a $X_k\in \TeichSpace(S)$ such that $E_{X_k}(\gamma)E_{X_k}(\eta) \leq i(\gamma,\eta)^2 + 1/k$ and $\abs{E_{X_k}(\gamma)-1}\leq 1/k$. Plugging these inequalities into \Cref{eq:lengthestimatefromenergy} yields
\[
\frac{1}{2} \frac{i(\gamma,\eta)^2(1-1/k)}{i(\gamma,\eta)^2+1/k} \cdot \ell_\rho^2(\gamma) \leq \tau^-(X_k,\gamma) \leq \tau^+(X_k, \gamma) \leq \frac{1}{2}(1+1/k) \cdot \ell_\rho^2(\gamma).
\]
It follows that both $\tau^-(X_k,\gamma)$ and $\tau^+(X_k,\gamma)$ converge to $\frac{1}{2}\cdot \ell^2_\rho(\gamma)$ for $k\to\infty$. We see that $\ell_\rho(\gamma)$ is entirely determined by the energy spectrum since the same holds true for the functions $\tau^+$ and $\tau^-$.

More precisely, if $\rho'$ is a second non-positively curved Riemannian metric on $S$ with equal energy spectrum, then \Cref{eq:lengthestimatefromenergy} also holds with $\ell_{\rho'}(\gamma)$ in place of $\ell_\rho(\gamma)$. We then see that 
\[
\frac{1}{2}\ell^2_{\rho'}(\gamma) =\lim_{k\to\infty} \tau^-(X_k,\gamma) = \lim_{k\to\infty} \tau^+(X_k,\gamma) = \frac{1}{2}\ell^2_{\rho}(\gamma)
\]
hence $\ell_{\rho}(\gamma) = \ell_{\rho'}(\gamma)$. Since $\gamma\subset S$ was arbitrary, it follows that $\rho$ and $\rho'$ have equal simple length spectrum.
\end{proof}
\section{Further comparison to the length spectra}\label{sec:furthercomparison}
In this section we show that the converse to the result of the previous section does not hold. Namely, the simple length spectrum does not determine the energy spectrum. Thus, we see that the energy spectrum carries more information.

\begin{proposition}\label{prop:lengthspectrumdoesnotgiveenergyspectrum}
For every hyperbolic metric on a surface there exists a negatively curved Riemannian metric on that surface with equal simple length spectrum but different energy spectrum.
\end{proposition}

We will show this by proving that the energy spectrum encodes the area of a Riemannian metric on a surface, whereas the simple length spectrum does not. We will make use of the following well-known observation.

\begin{lemma}\label{lem:minimumenergyisarea}
Let $(S,\rho)$ be a surface of genus at least one equipped with a Riemannian metric. Then the energy spectrum of $(S,\rho)$ (with $[f]=[\id]$) satisfies
\[
\Energy(X) \geq \Area(S,\rho) \text{ for all } X\in \TeichSpace(S).
\]
If, furthermore, the metric $\rho$ is non-positively curved, then equality is achieved if and only if $X$ equals $[\rho] \in \TeichSpace(S)$, the point in Teichm\"uller space determined by the metric $\rho$.
\end{lemma}
\begin{proof}
Let $\sigma$ be a hyperbolic metric on $S$. The metrics $\sigma$ and $\rho$ determine conformal structures on $S$. In corresponding local conformal coordinates $z$ resp. $w$ on $S$ we can write $\sigma = \sigma(z) \abs{dz}^2$ and $\rho = \rho(w) \abs{dw}^2$. Then the energy density of a map $h \colon (S,\sigma) \to (S,\rho)$ is given by
\[
e(h;\sigma,\rho) = \frac{\rho(h(z))}{\sigma(z)}\left\lbrace \abs{h_z}^2 + \abs{h_{\overline{z}}}^2 \right\rbrace
\]
and its Jacobian is given by
\[
J(h;\sigma,\rho) = \frac{\rho(h(z))}{\sigma(z)}\left\lbrace \abs{h_z}^2 - \abs{h_{\overline{z}}}^2 \right\rbrace
\]
(see \cite[Section 2]{WolfTeichmullerTheory}). Integrating over $S$ gives
\[
\mathcal{E}(h;\sigma,\rho) = \int_{S} e(h;\sigma,\rho) \vol_\sigma \geq \int_{S} J(h;\sigma,\rho) \vol_\sigma = \Area(S,\rho)
\]
with equality if and only if $h$ is a conformal map (i.e. $h_{\overline{z}} = 0$).

From this follows immediately that $\Energy(X) \geq \Area(S,\rho)$ for all $X\in \TeichSpace(S)$. If $[\sigma]=[\rho]\in \TeichSpace(S)$, then there exists a conformal map $h\colon (S,\sigma) \to (S,\rho)$ homotopic to the identity. For this map we see that $\Energy(h;\sigma,\rho) = \Area(S,\rho)$, so $\Energy([\sigma]) = \mathcal{E}(h;\sigma,\rho) = \Area(S,\rho)$. 

Finally, suppose $\rho$ has non-positive curvature. Assume $X = [\sigma] \in \TeichSpace(S)$ such that $\Energy(X) = \Area(S,\rho)$. By \cite{EellsSampson} there exists a energy minimising harmonic map $h \colon (S,\sigma) \to (S,\rho)$ homotopic to the identity. Then $\mathcal{E}(h;\sigma,\rho) = \Energy(X) = \Area(S,\rho)$, hence $h$ must be a conformal map. Because $h$ has degree one, it follows from the Riemann-Hurwitz formula that it can not have branch points. We conclude that $h$ is a biholomorphism isotopic to the identity which means that $X = [\sigma] = [\rho]$.
\end{proof}

\begin{proof}[{\Cref{prop:lengthspectrumdoesnotgiveenergyspectrum}}]
Let $\rho$ be any hyperbolic metric on the surface $S$. Let $\mathcal{G}$ be the union of all simple closed geodesics in $(S,\rho)$. Birman and Series prove in \cite{BirmanSeries} that this set is nowhere dense on $S$. In particular there exists an open set $U\subset S$ such that $\overline{U}$ does not intersect $\overline{\mathcal{G}}$. Let $\chi \colon S \to [0,1]$ be a smooth bump function which is zero outside of $U$ and equals one on some point in $U$. For $\delta>0$ we consider the metric $\rho' = (1+\delta\cdot \chi)\rho$. If we take $\delta$ small enough, then $\rho'$ is still a negatively curved metric. Because $\rho = \rho'$ on an open neighbourhood of $\mathcal{G}$, it follows that the simple closed geodesics for either metric are the same. As a result their simple length spectra are equal.

Finally, on some points in $U$ we have that $(1+\delta\cdot \chi) > 1$ and hence $\Area(S,\rho')>\Area(S,\rho)$. Taking into consideration \Cref{lem:minimumenergyisarea} we see (denoting the energy spectra of $\rho$ and $\rho'$ by $\Energy$ and $\Energy'$ respectively) that
\[
\min_{X\in\TeichSpace(S)} \Energy'(X) = \Area(S,\rho') > \Area(S,\rho) = \min_{X\in\TeichSpace(S)} \Energy(X) 
\]
so $\Energy \neq \Energy'$.
\end{proof}

We conclude that the energy spectrum is a more sensitive way to tell non-positively curved Riemannian metrics on $S$ apart than the simple length spectrum. With this in mind, we can pose the following interesting question: how does the energy spectrum compare to the (full) marked length spectrum? 

The marked length spectrum carries much more information than the simple length spectrum. Namely, Otal proved in \cite{Otal} that the set of negatively curved Riemannian metrics on a surface, determined up to isotopy, satisfies marked length spectrum rigidity. Furthermore, in \cite{CrokeFathiFeldman}, it is proved that the same holds true for the set of non-positively curved Riemannian metrics under the additional assumptions that these metrics do not have conjugate points. It follows in particular that for such metrics the marked length spectrum determines the energy spectrum. A, to the author, interesting question is now whether the sensitivity of energy spectrum falls strictly between that of the simple length spectrum and full marked length spectrum or whether the energy spectrum can also distinguish between all non-positively curved Riemannian metrics.

Taking this one step further we mention that Bonahon showed in \cite{Bonahon} that when considering marked length spectrum rigidity one can not drop the assumption that the metrics under consideration are Riemannian. More precisely, for any Riemannian metric of negative curvature on $S$ he constructed a non-Riemannian metric that has the same marked length spectrum but that is not isometric by an isometry isotopic to the identity. The notion of Dirichlet energy can be generalised to maps between manifolds with non-Riemannian metrics (see \cite{KorevaarSchoen}) and hence also in this context the energy spectrum can be defined. This allows us to ask whether the energy spectrum could perhaps provide more information and distinguish between negatively or non-positively curved non-Riemannian metrics.

\section{Energy spectrum rigidity}\label{sec:rigidity}
We now consider the question whether the energy spectrum of a Riemannian metric uniquely determines that metric (up to isotopy). If $\mathcal{M}$ is a set of metrics on $S$, determined up to isotopy, then we can consider the map $\mathcal{M} \to C^0(\TeichSpace(S))$ mapping a metric to its energy spectrum. We say $\mathcal{M}$ satisfies \textit{energy spectrum rigidity} if this map is injective. In light of \Cref{thm:energyspectrumgiveslengthspectrum} we see that this question is closely related to the question which classes of metrics on surfaces satisfy simple length spectrum rigidity. We describe here some examples where energy spectrum rigidity does hold.

\subsection{Hyperbolic metrics}
We consider the set of hyperbolic metrics on $S$, defined up to isotopy. As discussed in \Cref{subsec:teichmullerspace} this is the Teichm\"uller space of $S$. The existence of the harmonic maps under consideration is in this case a consequence of \cite{EellsSampson}.

It follows from elementary considerations on harmonic maps between surfaces that $\TeichSpace(S)$ satisfies energy spectrum rigidity, even without invoking simple length spectrum rigidity. Namely, we see from \Cref{lem:minimumenergyisarea} that a point in Teichm\"uller space can be recovered from its energy spectrum by locating the unique minimum. 

\begin{corollary}\label{cor:hyperbolicmetricenergyspecrigidity}
The set of hyperbolic metrics on $S$, defined up to isotopy, satisfies energy spectrum rigidity.
\end{corollary}

\subsection{Singular flat metrics}
As described in \Cref{subsec:conformalgeometry} a quadratic differential on a surface induces a metric on that surface. Away from the zeroes of the quadratic differential these metrics are locally flat and at the zero points they have a cone singularity of cone angle $(2+p)\pi$, $p\in\N$ (for more information see \cite{DuchinLeiningerRafi}). We call such metrics \textit{singular flat metrics} on the surface. We consider the set $\mathcal{M}$ of singular flat metrics on the surface $S$ that are induced by quadratic differentials, up to isotopy. The space of quadratic differentials, and hence also $\mathcal{M}$, can be canonically identified with the cotangent bundle of $\TeichSpace(S)$.

In \cite{DuchinLeiningerRafi} Duchin, Leiniger and Rafi prove the following theorem.

\begin{theorem}[{\cite[Theorem 1]{DuchinLeiningerRafi}}]\label{thm:singularflatsimplelengthrigidity}
Let $\mathcal{M}_1 \subset \mathcal{M}$ be the set of singular flat metrics on $S$ with area one, defined up to isotopy. The set $\mathcal{M}_1$ satisfies simple length spectrum rigidity.
\end{theorem}

Combining this fact with \Cref{thm:energyspectrumgiveslengthspectrum} and \Cref{lem:minimumenergyisarea} easily gives the following corollary.

\begin{corollary}\label{cor:singflatmetricsenergyspecrigidity}
The set of singular flat metrics that are induced by quadratic differentials, defined up to isotopy, satisfies energy spectrum rigidity.
\end{corollary}

\begin{proof}
Let $\rho, \rho'\in \mathcal{M}$ be two singular flat metrics on $S$ with equal energy spectrum. \Cref{lem:minimumenergyisarea} gives $\Area(S,\rho)=\Area(S,\rho')$. Then the rescaled metrics $\rho/\Area(S,\rho)$ and $\rho'/\Area(S,\rho')$ lie in $\mathcal{M}_1$ and by \Cref{thm:energyspectrumgiveslengthspectrum} have equal simple length spectrum. It now  follows from \Cref{thm:singularflatsimplelengthrigidity} that there exists an isometry between $\rho$ and $\rho'$ that is isotopic to the identity.
\end{proof}

Let us mention that also in this case the energy infimum in the definition of the energy spectrum is always realised by a harmonic map. These are however not harmonic maps in the precise sense we defined above because singular flat metrics are not actual Riemannian metrics. However, a more general notion of harmonic map, allowing for maps into metric spaces, has been developed in \cite{KorevaarSchoen}. Theorem 2.7.1 of that paper yields the existence of harmonic maps into surfaces equipped with singular flat metrics. In order to apply this result we note that if $S$ is a surface of genus at least two equipped with a singular flat metric, then its universal cover is a metric space of non-positive curvature (in the sense of Alexandrov).

\section{Kleinian surface groups}
A \textit{Kleinian surface group} is a representation $\rho \colon \pi_1(S) \to \PSL(2,\C)$ that is discrete and faithful. Because $\PSL(2,\C)$ acts on $\H^3$ by isometries, given a Kleinian surface group $\rho$ we can consider the hyperbolic 3-manifold $N = \H^3/\rho(\pi_1(S))$. The representation $\rho$ induces an identification between $\pi_1(S)$ and $\pi_1(N)$. As a result there is a one-to-one correspondence between the free homotopy classes of loops in $S$ and those of loops in $N$. The translation length of an element $\rho(\gamma)$ ($\gamma\in \pi_1(S)$), denoted $\ell_\rho(\gamma)$, is defined to be the infimum of the lengths of loops in $N$ that lie in the free homotopy class determined by $\gamma$. If $\rho(\gamma)$ is a parabolic element, then $\ell_\rho(\gamma) = 0$. If $\rho(\gamma)$ is an hyperbolic element, then it is conjugate to a matrix of the form
\[
\begin{pmatrix}
\lambda & 0 \\
0 & \lambda^{-1}
\end{pmatrix}
\]
with $\lambda\in \C, \abs{\lambda}>1$. In this case
\begin{equation}\label{eq:translationlength}
\ell_\rho(\gamma) = 2 \log \abs{\lambda}.
\end{equation}
The simple length spectrum of a Kleinian surface group is the vector $(\ell_\rho(\gamma))_{\gamma\in \mathcal{S}}$.

The representation $\rho$ determines a unique homotopy class $[f]$ of maps from $S$ to $N$ that lift to $\rho$-equivariant maps $\widetilde{S} \to \H^3$. We define the energy spectrum of a Kleinian surface group to be the energy spectrum of the hyperbolic manifold $N = \H^3/\rho(\gamma)$ and the homotopy class $[f]$.

In this section we prove the following analogue to \Cref{thm:energyspectrumgiveslengthspectrum}.

\begin{theorem}\label{thm:energyspectrumgiveslengthspectrumkleinian}
Let $\rho, \rho' \colon \Gamma \to \PSL(2,\C)$ be two Kleinian surface groups. If the energy spectra of $\rho$ and $\rho'$ coincide, then their simple simple length spectra coincide.
\end{theorem}

Bridgeman and Canary prove in \cite[Theorem 1.1]{BridgemanCanary} that a Kleinian surface group is determined up to conjugacy by its simple length spectrum. Combining their result with \Cref{thm:energyspectrumgiveslengthspectrumkleinian} gives the following corollary.

\begin{corollary}\label{cor:kleiniansurfgroupenergyspecrigidity}
If $\rho, \rho' \colon \Gamma \to \PSL(2,\C)$ are Kleinian surface groups with equal energy spectrum, then $\rho'$ is conjugate to either $\rho$ or $\overline{\rho}$.
\end{corollary}

The proof detailed in \Cref{sec:theenergyspectrum} can largely be carried over to the case of Kleinian surface groups. We do, however, need a replacement for \Cref{lem:dehntwistlengthestimate}. This will be provided by the following lemma.

\begin{lemma}\label{lem:dehntwistlengthestimatekleinian}
Let $\rho \colon \Gamma \to \PSL(2,\C)$ be a Kleinian surface group. Let $\gamma,\eta \subset S$ be simple closed curves with $i(\gamma,\eta) \in \{1,2\}$. Then there exists a constant $C=C(\rho,\gamma,\eta)>0$ such that
\[
\ell_\rho(T^n_\gamma \eta) \geq n \cdot i(\gamma,\eta) \cdot \ell_\rho(\gamma) - C
\]
for all $n\geq 1$.
\end{lemma}
Our proof is along similar lines as \cite[Lemma 2.2]{BridgemanCanary}.
\begin{proof}
We first consider the case $i(\gamma,\eta) = 2$. Let us denote $\gamma\cap \eta = \{x_0,x_1\}$. We assume that $\gamma$ and $\eta$ are parametrised loops starting at $x_0$. If we take $x_0$ as the basepoint of the fundamental group, then we can consider $\gamma$ and $\eta$ as elements of $\pi_1(S,x_0)$. We denote by $\gamma_1$ and $\eta_1$ the subarcs of $\gamma$ and $\eta$ respectively that connect $x_0$ to $x_1$ and we denote by $\gamma_2$ and $\eta_2$ the subarcs connecting $x_1$ to $x_0$ (see \Cref{fig:dehntwistdiagram}).

\begin{figure}[ht]
\centering
{
	\resizebox{80mm}{!}
	{\huge
		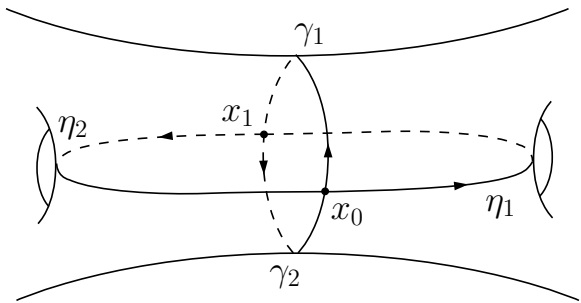
	}
}
\caption{Overview of the positions of the arcs $\gamma_1, \gamma_2, \eta_1$ and $\eta_2$.}\label{fig:dehntwistdiagram}
\end{figure}

We now find the following expression for the element $T^n_\gamma \eta \in \pi_1(S, x_0)$,
\begin{align*}
T^n_\gamma \eta &= \eta_2 (\gamma_2^{-1} \gamma_1^{-1})^n \eta_1 \gamma^n\\
&= \eta_2 \gamma_1 (\gamma_1^{-1} \gamma_2^{-1})^n \gamma_1^{-1} \eta_1 \gamma^n\\
&= \sigma \gamma^{-n} \nu \gamma^n
\end{align*}
where we put $\sigma = \eta_2 \gamma_1, \nu = \gamma_1^{-1} \eta_1 \in \pi_1(S,x_0)$.

We note that if $\rho(\gamma)$ is a parabolic element, then $\ell_\rho(\gamma) = 0$ and the statement is trivial. Hence, from now on we assume $\rho(\gamma)$ is a hyperbolic element. By conjugating the representation $\rho$ we can assume that, for some $\lambda\in \C, \abs{\lambda} > 1$, we have
\[
\rho(\gamma) = \begin{pmatrix}
\lambda & 0 \\
0 & \lambda^{-1}
\end{pmatrix}.
\]
Note that a matrix representing an element of $\PSL(2,\C)$ is only determined up to a multiplication by $\pm \id$. However, for our calculation of the translation length this does not matter.

For suitable coefficients $a,b,c,d,a',b',c',d'\in \C$ we can write 
\[
\rho(\sigma) = \begin{pmatrix}
a & b\\
c & d
\end{pmatrix} \text{ and } \rho(\nu) = \begin{pmatrix}
a' & b'\\
c' & d'
\end{pmatrix}.
\]
We note that coefficients of these matrices do not vanish.  Namely, if a coefficient of, say, $\rho(\sigma)$ vanishes, then it maps a fixed point of $\rho(\gamma)$ to a fixed point of $\rho(\gamma)$. Then $\rho(\sigma \gamma \sigma^{-1})$ and $\rho(\gamma)$ share a fixed point which implies they must have a common power because $\rho(\Gamma)$ is discrete. Because the elements $\gamma$ and $\sigma\gamma\sigma^{-1}$ do not have a common power this would contradict that the representation $\rho$ is faithful.

A simple calculation yields that
\[
\rho(T^n_\gamma\eta) = \rho(\sigma\gamma^{-n} \nu \gamma^n) = 
\begin{pmatrix}
a a' + \lambda^2 b c' & b d' + \lambda^{-2} a b' \\
c a' + \lambda^2 d c' & d d' + \lambda^{-2} c b'
\end{pmatrix}.
\]
Now if $\left(\begin{smallmatrix} \alpha & \beta\\ \gamma & \delta \end{smallmatrix}\right) \in \SL(2,\C)$, then its eigenvalues are given by
\[
\mu_{\pm} = \frac{\alpha+ \delta}{2} \pm \frac{1}{2} \sqrt{(\alpha+\delta)^2 - 4}.
\]
Applying this to $\rho(T^n_\gamma\eta)$ (that is, taking $\alpha = a a' + \lambda^2 b c'$ and $\delta = d d' + \lambda^{-2} c b'$) we find that
\[
\mu_+ = \lambda^{2n}(b c' + O(\abs{\lambda}^{-2n})).
\]
Using \Cref{eq:translationlength} and the fact that $bc' \neq 0$ gives
\begin{align*}
\ell_\rho(T^n_\gamma\eta) &= 2 \log \abs{\mu_+} = 4 \cdot n \cdot \log \abs{\lambda} + \log(\abs{bc' + O(\abs{\lambda}^{-2n})}) \\&= 2 \cdot n \cdot \ell_\rho(\gamma) + O(1) = i(\gamma,\eta) \cdot n \cdot \ell_\rho(\gamma) + O(1) \text{ as } n\to\infty
\end{align*}
This proves the lemma for the case $i(\gamma,\eta) = 2$. In the case $i(\gamma,\eta)=1$ we have that $T^n_\gamma\eta = \eta \gamma^n$. The calculation of the largest eigenvalue of $\rho(\eta\gamma^n)$ is similar and is carried out in \cite[Lemma 2.2]{BridgemanCanary}. Filling the formula of that lemma into \Cref{eq:translationlength} immediately gives the result also in this case.
\end{proof}
We can now give a proof of \Cref{thm:energyspectrumgiveslengthspectrumkleinian}.
\begin{proof}[Proof of \Cref{thm:energyspectrumgiveslengthspectrumkleinian}]
The proof of \Cref{thm:energyspectrumgiveslengthspectrum} goes through in the present situation mostly unchanged. Let us only point the modifications that need to be made. In this proof we denote by $[f]$ the homotopy class of maps $S\to N$ that lift to a $\rho$-equivariant map $\widetilde{S} \to \H^3$.

First we consider the proof of \Cref{lem:tauupperbound}. Let $\gamma \in \pi_1(S)$ be an element that corresponds to a simple closed curve. If $\rho(\gamma)$ is hyperbolic, then there exists a length minimising geodesic loop $\eta \colon \R/\Z \to N$ in the free homotopy class determined by $\gamma$. By deforming a map in $[f]$ we can construct a Lipschitz continuous map $k_0 \colon S\to N$ such that $k_0\in [f]$ and $k_0(x,[y]) = \eta([y])$ on $A_\epsilon$ (notation as in the proof of \Cref{def:taunotation}). The maps $k_n$ can then be constructed as before and the energy estimates also go through. We find that $\tau^+(X,\gamma) \leq \frac{1}{2}E_X(\gamma) \cdot \ell_\rho^2(\gamma)$.

If $\rho(\gamma)$ is a parabolic element, then no such geodesic loop exists. However, since $\ell_\rho(\gamma) = 0$ there exists for every $\delta>0$ a closed loop $\eta \colon \R/\Z \to N$ with $l(\eta) \leq \delta$. If we then take a map $k_0 \colon S \to N$ in the homotopy class $[f]$ with $k_0(x,[y]) = \eta([y])$ on $A_\epsilon$ and carry out the rest of the argument of the proof of proof of \Cref{lem:tauupperbound} we find
\[
\tau^+(X,\gamma) \leq \frac{1}{2} E_X(\gamma) \cdot l^2(\eta) \leq \frac{1}{2} E_X(\gamma) \cdot \delta^2.
\]
Since $\delta$ was arbitrary $\tau^+(X,\gamma) = \frac{1}{2} E_X(\gamma) \cdot \ell_\rho^2(\gamma) = 0$ follows.

Let us now consider the proof of \Cref{lem:taulowerbound}. Suppose $\gamma, \eta \in \pi_1(S)$ correspond to simple closed curves with $i(\gamma,\eta) \in \{1,2\}$. Any map in $[f\circ T_\gamma^n]$ maps the curve $\eta$ to a curve in the free homotopy class determined by $T^n_\gamma\eta$. The results of \Cref{lem:dehntwistlengthestimatekleinian} and \Cref{lem:minskyestimate} then give rise to the estimate
\[
\tau^-(X,\gamma) \geq \frac{1}{2} \frac{i(\gamma,\eta)^2 \cdot \ell_\rho^2(\gamma)}{E_X(\eta)} 
\]
in the same way as in the proof of \Cref{lem:taulowerbound}.

It follows that the estimates of \Cref{eq:lengthestimatefromenergy} are also true in the present situation whenever $i(\gamma,\eta) \in \{1,2\}$. Because the curves $\gamma$ and $\eta$ constructed in \Cref{lem:almostcollarequality} do satisfy this condition we see that the remainder of the proof of \Cref{thm:energyspectrumgiveslengthspectrum} can now be followed verbatim.
\end{proof}

\section{Hitchin representations}\label{sec:hitchinrepresentationsenergyspectrum}
A \textit{Hitchin representation} is a representation $\rho \colon \pi_1(S) \to \PSL(n,\R)$ that lies in a particular connected component (discovered by Hitchin in \cite{HitchinLieGroups}) of the representation variety $\Rep(\pi_1(S), \PSL(n,\R))$. Such representations are discrete, faithful (\cite{LabourieAnosovFlows}) and act isometrically on the symmetric space $\PSL(n,\R)/\PSO(n)$. It follows that their simple length spectrum and energy spectrum can be defined in the same manner as in the previous section. 

As stated in the introduction our main interest is the study of the energy spectrum for Hitchin representations. Unfortunately, the methods presented here are not sufficient to conclude that a Hitchin representation is uniquely determined by its energy functional. Let us briefly describe the difficulty we encounter.

The author believes that an analogue of \Cref{lem:dehntwistlengthestimatekleinian} holds also for Hitchin representations. Then the proof presented in the previous section can be carried out for Hitchin representations. Hence, it seems likely that their simple length spectrum is also determined by their energy spectrum. However, it is not known to the author whether a Hitchin representation is determined by its simple length spectrum (as we define it here). 

Let us point out that very closely related results are obtain by Bridgeman, Canary and Labourie in \cite{BridgemanCanaryLabourie}. Namely, they prove that Hitchin representations are rigid for a different type of simple length spectrum\footnote{Note that in \cite{BridgemanCanaryLabourie} the term `simple length spectrum' is also used; however, it does not refer to the same quantity we consider here.}. Let us briefly describe the difference. If $\gamma\in \pi_1(S)$, then the $\rho(\gamma)$ is a diagonalisable matrix with real eigenvalues (which are determined up to sign). Denote these by $\lambda_1,\hdots, \lambda_n$. Then the spectral length of $\rho(\gamma)$ is $L_\rho(\gamma) = \max_{i=1,\hdots,n} \abs{\lambda_i}$ and its trace is $\abs{\text{Tr}(\rho(\gamma))} = \sum_{i=1}^n \abs{\lambda_i}$. In \cite{BridgemanCanaryLabourie} it is proved that a Hitchin representation is determined, up to conjugacy, by its simple (spectral) length spectrum $(L_\rho(\gamma))_{\gamma\in \mathcal{S}}$ and by its simple trace spectrum $(\abs{\text{Tr}(\rho(\gamma))})_{\gamma\in \mathcal{S}}$. In contrast, the simple length spectrum we consider in this paper assigns to each simple closed curve $\gamma$ the translation length of $\rho(\gamma)$, which is given by $\ell_\rho(\gamma) = \sqrt{\sum_{i=1}^n \abs{\lambda_i(\rho(\gamma))}^2}$. So in order to finish the circle of ideas presented in this paper it remains to answer the question whether a Hitchin representation is determined, up to conjugacy, by its simple (translation) length spectrum.

\bibliographystyle{alpha}
\bibliography{bibliography}

\end{document}